\documentclass[11pt]{amsart}
\usepackage[arrow,matrix,graph,frame,poly,arc,tips]{xy}
\usepackage{amscd,amssymb,subfigure,epsfig,color}
\usepackage[dvipdfm]{hyperref}
\usepackage[width=5.65in,height=8.1in,centering]{geometry}
\usepackage{amsmath}
\usepackage{eepic}
\usepackage{fancyvrb}

\theoremstyle{plain}
\newtheorem{theorem}[subsection]{Theorem}

\newtheorem{lemma}[subsection]{Lemma}
\newtheorem{proposition}[subsection]{Proposition}
\newtheorem{corollary}[subsection]{Corollary}

\theoremstyle{definition}
\newtheorem{remark}[subsection]{Remark}
\newtheorem{definition}[subsection]{Definition}
\newtheorem{example}[subsection]{Example}

\newtheorem*{ack}{Acknowledgment}

\numberwithin{equation}{section}

\newcommand{\abs}[1]{{\lvert #1 \rvert}}
\newcommand{\set}[1]{{\left\{ #1 \right\}}}

\newcommand{\A}{{\mathcal A}}
\newcommand{\B}{{\mathcal B}}
\newcommand{\cC}{{\mathcal C}}

\newcommand{\cF}{{\mathcal F}}

\newcommand{\FF}{{\mathbf F}} % flag of faces
\newcommand{\XX}{{\mathbf X}} % flag of flats

\def\cprime{$'$}

\newcommand{\R}{{\mathbb R}}

\newcommand{\res}{{\rm res}}

\DeclareMathOperator{\flag}{Flag}
\DeclareMathOperator{\Fl}{Fl}

\DeclareMathOperator{\gr}{gr}

\def\px{{\mathchar"002B}} % +, not as an operator, just a symbol
\def\mx{{\mathchar"0200}} % -; same
\def\lx{{\mathchar"013C}} % <; same
\def\vareps{{\varepsilon}}

\title{Eigenvectors for a random walk on a hyperplane arrangement}
\date{\today}
\author[G. Denham]{Graham Denham$^1$}
\address{Department of Mathematics, University of Western Ontario\\
London, ON  N6A 5B7, Canada}
\urladdr{\href{http://www.math.uwo.ca/~gdenham}%
{http://www.math.uwo.ca/\char'176gdenham}}
\thanks{{$^1$}Partially supported by NSERC of Canada and the Swiss National
Science Foundation}
\begin{document}
\begin{abstract}
We find explicit eigenvectors for the transition matrix of the 
Bidegare-Hanlon-Rockmore random walk, from \cite{BHR99}.  This is
accomplished by using Brown and Diaconis' analysis in \cite{BD98}
of the stationary distribution, together with some combinatorics
of functions on the face lattice of a hyperplane arrangement, due 
to Gel\cprime fand and Varchenko~\cite{VG87}.
\end{abstract}
\maketitle
\section{Introduction}
In 1999, Bidegare, Hanlon and Rockmore~\cite{BHR99} introduced a 
simultaneous generalization of several well-studied discrete Markov chains.
Let $\A$ be an arrangement of $n$ linear
hyperplanes in $W=\R^\ell$, and let $\cC$ denote the set of chambers:
i.e., the connected components of the complement of $\A$ in $W$. 
They construct a random walk on $\cC$, which we will follow \cite{BD98}
in calling the BHR random walk, by means of the face product,
defined below.  By choosing the hyperplane arrangement suitably, 
one obtains as special cases the Tsetlin library (``move-to-front rule''), 
Ehrenfests' urn, and various card-shuffling models~\cite{BHR99,BD98}.

A key insight in this construction and main result of \cite{BHR99}
is that the eigenvalues of the transition matrix 
can be expressed simply in terms of the combinatorics
of the arrangement $\A$.  This is useful for bounding the rate of 
convergence to the stationary distribution: Brown and Diaconis
make this analysis in \cite{BD98}, show that the transition matrix is
diagonalizable, and give an explicit
description of the random walk's stationary distribution.

Subsequently Brown~\cite{Br00} generalized the BHR random walk
to the setting of certain semigroup algebras, showing in particular
that the diagonalization result held there as well.  He 
describes projection operators onto each eigenspace, which in principle
provides a description of the eigenvectors of the random walk's 
transition matrix.  However, the general expression is necessarily somewhat
complicated.

The purpose of this paper, then, is to provide a relatively straightforward
description of the original BHR random walk's eigenvectors.  It turns out
that the combinatorial Heaviside functions of Gel\cprime fand and 
Varchenko~\cite{VG87} behave well with respect to the face product and
the BHR random walk, so we highlight their role in this problem.
Using the description of the stationary distribution from \cite{BD98},
the main result here, Theorem~\ref{th:main}, describes a spanning set for
each eigenspace in terms of flags in the intersection lattice.
\section{Background and notation}
\subsection{The face algebra}\label{ss:face}
Let $\A$ be an arrangement of $n$ hyperplanes in $\R^\ell$.  
We will assume throughout that $\A$ is central and essential: that
is, the intersection of the hyperplanes $\A$ equals the origin in $\R^\ell$.
We will follow the notational conventions of the standard reference
for hyperplane arrangements, \cite{OTbook}.  In particular, let $L(\A)$ 
denote the lattice of intersections, ordered by reverse inclusion, and
$\cF=\cF(\A)$ the face semilattice, also ordered by reverse inclusion.
For $0\leq p\leq \ell$, let $L_p(\A)$ and $\cF_p(\A)$ denote the subspaces
and faces, respectively, of codimension $p$.  In particular, the set of
chambers $\cC=\cC(\A)=\cF_0(\A)$.

For a face $F\in\cF$, let $\abs{F}$ denote the smallest subspace in 
$L(\A)$ containing $F$.  Recall that $\A^X$ denotes the arrangement 
in $X$ of hyperplanes $\set{H\cap X\colon H\in\A,X\not\subseteq H}$,
whenever $X\in L(\A)$, and $\A_X$ denotes the subarrangement
$\set{H\in\A\colon X\subseteq H}$.  We will identify faces (and chambers) of
$\A^X$ with faces $F\in\cF(\A)$ for which $\abs{F}\geq X$ (and
$\abs{F}=X$, respectively):
\begin{equation}\label{eq:loc}
\cF(\A^X)\cong \set{F\in\cF(\A)\colon\abs{F}\geq X}.
\end{equation}
For hyperplanes $H\in \A$, let $F_H$ be $0,\pm 1$
depending on 
whether $F$ is contained in $H$, on the positive side, or the 
negative side, respectively.  We will abbreviate the values of $F_H$ with
$0,\px,\mx$.  A face $F$ is uniquely determined
by the sign sequence $(F_H)_{H\in\A}$.

We recall the face product of two faces $F$ and $G$
can be described by its sign sequence:
\begin{equation}\label{eq:vprod}
(FG)_H=\begin{cases}
F_H & \text{if $F_H\neq0$;}\\
G_H & \text{otherwise.}
\end{cases}
\end{equation}
From the definition, $FF=F$ for any $F$, and $FGF=FG$ for any faces $F,G$
(the ``deletion property'': see \cite{Br00}.)

For a fixed arrangement $\A$, let $A=R[\cF]$ denote its face algebra, 
introduced in \cite{BD98}.  Additively, $A$ is the free $R$-module on
$\cF$, and multiplication is given by linearly extending the face product.
Following \cite{BD98}, let $V(\A)$ denote the free $R$-module on $\cC(\A)$:
then the face product makes $V(\A)$ a left $A$-module.  More generally,
following \cite[Section~5C]{BD98}, we can make $V(\A^X)$ a left $A$-module
for any $X\in L(\A)$: for $F\in\cF$ and $G\in\cC(\A^X)$, set
\begin{equation}\label{eq:localaction}
FG=\begin{cases}
FG & \text{if $\abs{F}\geq X$;}\\
0 & \text{otherwise.}
\end{cases}
\end{equation}

\subsection{The zonotope of a real arrangement}
\begin{figure}
\includegraphics[height=1.5in]{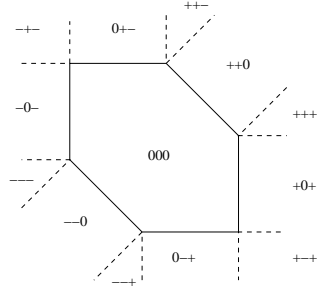}
\caption{Zonotope for $f_1=x$, $f_2=y$, $f_3=x-y$.}\label{fig:one}
\end{figure}
As Brown and Diaconis~\cite{BD98} observe, 
it is useful in this setting
to consider the zonotope of a real arrangement.  Let
$f_H$ denote a linear equation defining the hyperplane $H$, for each
$H\in\A$.  The zonotope $Z_\A$ is, by definition, the Minkowski sum
of intervals,
\[
Z_\A=\sum_{H\in\A}[-f_H,f_H].
\]
This is a polytope in $W^*$: we refer to \cite{Ziebook} for details.
Its most relevant feature here
is that the face poset of $Z_\A$ (ordered by inclusion)
is isomorphic to $\cF(\A)$: the isomorphism arises by identifying the
outer normal fan of the zonotope with the arrangement.  If $F$ is a 
face of $\A$, then, let $\hat{F}$ denote the corresponding face of the
zonotope.  Each face $\hat{F}$ is an affine translate of a
zonotope of a closed subarrangement: for $F\in\cF$, let $X=\abs{F}$ and
$v_F=\sum_{H\in\A}(F_H)f_H$.  Then $\hat{F}=v_F+Z_{\A_X}$: 
see, for example, discussion in \cite{McM71}.

For the sake of intuition, we indicate one way to visualize the face
product in the zonotope world.
\begin{definition}\label{def:retract}
Suppose $P$ is a polytope in $W^*$ and $v\in W^*$ is nonzero.  Define a
map $r_{P,v}\colon P+[-v,v]\to P+v$ as follows.  If $p\in P+[-v,v]$,
let $\lambda$ be maximal for which $p=x+\lambda v$, where $x\in P$ 
and $\lambda\leq 1$.  Set $r_{P,v}(p)=x+v$.
Clearly $r_{P,v}$ is piecewise-linear and fixes $P+v$ pointwise.
\end{definition}
\begin{proposition}\label{prop:Zprod}
For each face $F$ of $\A$, there is a piecewise-linear retract $p_F$
of $Z_\A$ onto $\hat{F}$, with $p_F(\hat{G})=\widehat{FG}$, for all
$G\in \cF(\A)$.  
\end{proposition}
\begin{proof}
We will write $p_F$ as the composition of the retracts from 
Definition~\ref{def:retract}, 
one for each hyperplane of $\A$ not containing $F$.  

Again, let $X=\abs{F}$.  Let $S_1=\set{f_H\colon F_H=0}$ and
$S_2=\set{(F_H)f_H\colon F_H\neq 0}$.  Then 
\begin{eqnarray*}
Z_\A&=&\sum_{v\in S_1}[-v,v]+\sum_{v\in S_2}[-v,v]\\
&=&Z_{\A_X}+\sum_{v\in S_2}[-v,v],
\end{eqnarray*}
and $v_F=\sum_{v\in S_2}v$.  Put the elements of $S_2$ in any order,
writing $S_2=\set{v_1,v_2,\ldots,v_k}$.  Let $P_0=Z_{\A_X}+v_F$, and let
$P_i=P_{i-1}+[-2v_i,0]$ for $1\leq i\leq k$.  
Then 
\[
\hat{F}=P_0\subseteq P_1\subseteq\cdots\subseteq P_k=Z_\A,
\]
and the following composite collapses along
the vectors in $S_2$, one at a time:
\[
p_F=r_{P_0-v_1,v_1}\circ r_{P_1-v_2,v_2}\circ\cdots\circ r_{P_{k-1}-v_k,v_k}
\colon Z_\A\to \hat{F}.
\]
By construction, this is a piecewise-linear retract onto $\hat{F}$.  

To see that $p_F(\hat{G})=\widehat{FG}$ for any face $G$, let $S\in(\pm1)^\A$
be an arbitrary sequence of nonzero signs, 
and define $v_S=\sum_{H\in\A}S_H f_H$.  Then
$
r_{P_{i-1}-v_i,v_i}(v_S)=v_{S'},
$
where 
\[
(S')_H=\begin{cases} S_H&\text{if $f_H\neq \pm v_i$;}\\
c & \text{if $f_H=c v_i$.}\end{cases}
\]
It follows that $p_F(\hat{C})=\widehat{FC}$ for chambers $C\in\cC(\A)$.
By writing an arbitrary face $\hat{G}$ as the convex hull of the 
vertices $\hat{C}$
for which $C\leq G$, one then obtains $p_F(\hat{G})=\widehat{FG}$ for all $G$.
\end{proof}
Note that the map $p_F$ is not uniquely defined, and depends on a choice of
order.  See Figure~\ref{fig:retracts} for examples.
\begin{figure}
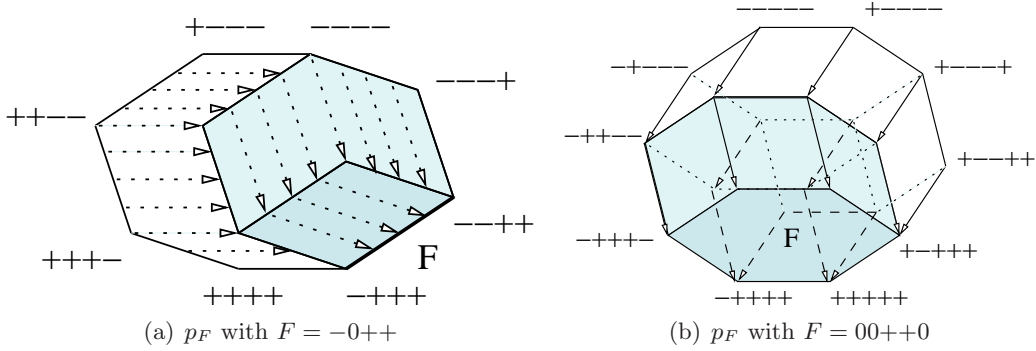

\centering
\subfigure[$p_F$ with $F=\mx0\px\px$]{
\includegraphics[height=1.5in]{zono6.pstex}
}
\subfigure[$p_F$ with $F=00\px\px0$]{
\includegraphics[height=1.6in]{zono3d.pstex}
}
\caption{Retracts onto faces of zonotopes of dimension $2$ and $3$}
\label{fig:retracts}
\end{figure}

\subsection{The BHR random walk}
Let $\A$ be an arrangement, and $w\colon\cF\to\R$ a discrete probability
distribution on the faces of $\A$.  The random walk introduced in
\cite{BHR99} is a random walk on chambers, taking $C\in\cC$ to $FC$ with
probability $w_F$.  Its transition matrix $K=K_\A$ can be regarded as a 
linear endomorphism of the vector space with basis $\cC$, given by
\begin{equation}\label{eq:BHR}
K(C)=\sum_{F\in\cF}w_FFC.
\end{equation}
An eigenvector of $K$ with eigenvalue $1$ gives a stationary distribution
for the random walk.  Brown and Diaconis~\cite[Theorem~2]{BD98} find that,
with an assumption of nondegeneracy, the eigenspace
is $1$-dimensional, which is to say 
the stationary distribution is unique.  In this
case, it is given by sampling faces without replacement: explicitly, one
sums over all permutations of the faces to obtain
\begin{equation}\label{eq:stationary}
\pi_C=\sum_{\substack{\sigma\in S_N\colon \\
C=F_{\sigma(1)}\cdots F_{\sigma(N)}}} \prod_{1\leq p\leq N}
\frac{w_{F_{\sigma(p)}}}{1-\sum_{i<p}w_{F_{\sigma(i)}}},
\end{equation}
where $N=\abs{\cF}$, the number of faces.  The complete list of eigenvalues
is given as follows, where $\mu$ denotes the M\"obius function of $L(\A)$. 
For each $X\in L(\A)$, let
\begin{equation}\label{eq:lambda}
\lambda_X=\sum_{F\colon\abs{F}\geq X} w_F.
\end{equation}
(Equivalently, $\lambda_X=\sum_{F\in\cF(\A^X)} w_F$, by \eqref{eq:loc}.)
\begin{theorem}[\cite{BHR99,BD98}]\label{th:BHR}
For each $X\in L_p(\A)$, for $0\leq p\leq \ell$, 
the matrix $K$ has an eigenspace of multiplicity
$(-1)^p\mu(W,X)$ with eigenvalue $\lambda_X$.
\end{theorem}

The main result here is a corresponding basis of eigenvectors for each
eigenvalue, Theorem~\ref{th:main}.  For this, it is convenient to
regard the distribution weights $w_F$ as indeterminates as in \cite{Br00}, 
diagonalize, and then specialize afterwards: let 
$R=\R(w_F\colon F\in \cF)$, the fraction field 
of polynomials in the variables $w_F$.  In doing so, we relax the
condition that the weights sum to $1$, so \eqref{eq:stationary}
needs to be adjusted accordingly.  Let $q\in V(\A)$ be the vector
whose coordinate on chamber $C$ is given by
\begin{equation}\label{eq:1eivec}
q_C=\sum_{\substack{\sigma\in S_N\colon \\
C=F_{\sigma(1)}\cdots F_{\sigma(N)}}} \prod_{p=1}^N
\big(\sum_{i=p}^N w_{F_{\sigma(i)}}\big)^{-1}.
\end{equation}
\begin{lemma}\label{lem:eivec}
For any arrangement $\A$, the 
vector $q$ is a $\lambda_W$-eigenvector of $K$.
\end{lemma}
\begin{proof}
We provide a direct calculation in lieu of adapting the corresponding
result from \cite{BD98}.  For succinctness, let
\[
f(x_1,\ldots,x_N)=\prod_{i=1}^N\big(\sum_{i=p}^N x_i\big)^{-1}
\]
for any choice of $x_i$'s, and abbreviate $f(\sigma):=f(w_{F_{\sigma(1)}},
\ldots,w_{F_{\sigma(N)}})$ for any permutation of the faces $\sigma\in S_N$.
By clearing denominators, one may verify the identity
\begin{align}\label{eq:identity}
f(x_1,\ldots,x_N)+&f(x_2,x_1,x_3,\ldots,x_N)+\cdots\nonumber \\
&+f(x_2,\ldots,x_N,x_1)=f(x_1,\ldots,x_N)\cdot\big(\sum_{i=1}^N x_i\big)/x_1.
\end{align}
For $1\leq i\leq N$, let $\sigma_i$ denote the $i$-cycle 
$(1,2,\ldots,i)\in S_N$.  Then \eqref{eq:identity} states that
\begin{equation}\label{eq:identity2}
\sum_{i=1}^N f(\sigma\sigma_i^{-1})=(\lambda_W/x_1)f(\sigma),
\end{equation}
since $\lambda_W=\sum_{i=1}^N w_{F_i}$.

Now $Kq= \sum_{F\in\cF} w_Fq_C(FC)$.  For each $C\in \cC$, we compute:
\begin{eqnarray*}
(Kq)_C&=& \sum_{F,C'\colon C=FC'} w_Fq_{C'}\\
&=&
\sum_{\substack{F\in\cF,\;\sigma\in S_N\colon 
\\ C=FF_{\sigma(1)}\cdots F_{\sigma(N)}}}
w_F f(\sigma)\\
&=& \sum_{\substack{\sigma\in S_N\colon \\
C=F_{\sigma(1)}\cdots F_{\sigma(N)}}}
\sum_{i=1}^N w_{F_{\sigma(1)}}f(\sigma\sigma_i^{-1}),\text{~by the deletion property,
\S\ref{ss:face},}\\
&=&\lambda_W q_C,\text{~by \eqref{eq:identity2}.}
\end{eqnarray*}
\end{proof}
\begin{remark}
As written, the weights $\set{w_F}$ in \eqref{eq:1eivec} only admit specializations
to nonzero real numbers.  One may clear denominators to obtain a general
polynomial expression for an eigenvector, however this is clumsy to write 
in general.
\end{remark} 

\subsection{Combinatorial Heaviside functions}
We will recover Theorem~\ref{th:BHR}, together with eigenvectors, by
exploiting some fundamental structural results of Varchenko and 
Gel\cprime fand~\cite{VG87}, which we briefly describe here.  For 
more details, see \cite{De02}. 

The Varchenko-Gel\cprime fand ring of $\A$ is defined additively to
be simply the space of linear functionals on chambers, $V^*$.  The
ring structure is given by coordinatewise multiplication.
The interest lies in the choice of generators: for each hyperplane
$H\in\A$, let $x_H\in V^*$ be the function defined by
\[
x_H(C)=\begin{cases}
1& \text{if $C$ is on the positive side of $H$;}\\
0& \text{otherwise.}
\end{cases}
\]
For a set of hyperplanes $I\subseteq\A$, let $x_I$ be the monomial
$x_I=\prod_{H\in I} x_H$.  (Since each $x_H$ is idempotent, 
we only need to consider square-free monomials.)  Let $1\in V^*$ be the
function given by $1(C)=1$ for all chambers $C\in\cC$.

In \cite{VG87}, it is shown that the Varchenko-Gel\cprime fand ring
admits a presentation much like the Orlik-Solomon algebra, with 
generators $x_H$ and certain combinatorial relations: see 
\cite{Pr06} for an interpretation that compares the two.
Unlike the Orlik-Solomon algebra, however, 
the relations amongst the generators are inhomogeneous, so it is
useful to define a degree filtration by letting
\[
P_pV^*=\set{f\in V^*\colon \text{$f$ can be written as a polynomial in 
$x_H$'s of degree at most $p$.}}
\]
By \cite[Theorem~1]{VG87}, 
\[
V^*=P_\ell V^*\supseteq P_{\ell-1} V^*\supseteq\cdots P_0 V^*\supsetneq 0,
\]
with the function $1$ spanning $P_0$.  The filtration is ``natural'' in
the sense that, if $\B$ is a subarrangement of $\A$, containment
gives a map of chambers $\cC(\A)\to \cC(\B)$.  This induces a map
$V(\B)^*\to V(\A)^*$, which is easily seen to preserve the degree filtration.

Now let $\gr_p V^*=P_p V^*/P_{p-1} V^*$, for $0\leq p\leq \ell$. 
Let $b_p=(-1)^p\sum_{X\in L_p(\A)}\mu(W,X)$, the $p$th Betti number of 
the arrangement.  Then the degree filtration satisfies an analogue of
Brieskorn's Lemma for arrangements:
\begin{theorem}[Theorem~3, Corollaries 2,3 in \cite{VG87}]\label{th:vg}
For $0\leq p\leq \ell$, $\gr_p V^*$ is a free module, of rank equal
to $b_p(\A)$.  More precisely, there are isomorphisms
\begin{equation}\label{eq:brieskorn}
\gr_p V(\A)^*\cong\bigoplus_{X\in L_p(\A)}\gr_p V(\A_X)^*.
\end{equation}
induced by the inclusion of the arrangement $\A_X$ into $\A$.
\end{theorem}

\begin{example}\label{ex:threelines}
Consider the arrangement of lines $\set{x,y,x-y}$ through the origin in $\R^2$,
as in Figure~\ref{fig:one}.
Call the lines $H_1,H_2,H_3$ in this order.  Bases for $\gr_p V^*$ are:
\[
\begin{array}{r|l}
p=2 & x_{H_1}x_{H_2},\;x_{H_1}x_{H_3},\\
p=1 & x_{H_i}\colon 1\leq i\leq 3,\\
p=0 & 1.
\end{array}
\]
For example,  
the function $x_{H_1}x_{H_3}$ takes the value $1$ on the chamber $x>0$, $x>y$,
and zero elsewhere.
\end{example}
\subsection{The dual filtration}\label{ss:dualfiltr}
The degree filtration defines an orthogonal,
decreasing filtration on the dual space,
$V^{**}\cong V$: following \cite{VG87}, let
\[
W^pV=\set{v\in V\colon f(v)=0\text{~for all $f\in P_{p-1}V^*$.}}
\]
Then 
\[
V=W^0V\supseteq W^1V\supseteq\cdots\supseteq W^{\ell+1}V=0.
\]
If $\B$ is a subarrangement of $\A$ and $i\colon \cC(\A)\to
\cC(\B)$ is the induced map of chambers, the ``natural'' map 
$V(\A)\to V(\B)$ extends $i$ linearly.  Dually, this map
preserves the $W$-filtration.

Let $\gr^pV=W^pV/W^{p+1}V$, for $0\leq p\leq\ell$.  The dual version
of Theorem~\ref{th:vg} reads as follows:
\begin{proposition}
For $0\leq p\leq \ell$, we have $\gr_p V^*\cong(\gr^p V)^*$.
So $\gr^p V$ is also free of rank $b_p$, and admits a decomposition
\begin{equation}\label{eq:brieskorn'}
\gr^p V(\A)\cong\bigoplus_{X\in L_p(\A)}\gr^p V(\A_X).
\end{equation}
\end{proposition}
\begin{proof}
We prove the first assertion, from which the rest follows directly.
Let $e\colon V\to V^{**}$ be the natural isomorphism.
If $x\in W^pV$, then $e(x)$ restricts to a map $P_pV^*\to R$ by $x(f)=f(x)$.
If $e(x)=0$, this means $f(x)=0$ for all $f\in P_pV^*$, so $x\in W^{p+1}V$.

On the other hand, restriction of functions gives a surjective map
$\res\colon (P_pV^*)^*\to (P_{p-1}V^*)^*$.  Putting this together, we have
shown that the sequence
\begin{equation}\label{eq:dualityseq}
\xymatrix{
0\ar[r] & W^{p+1}V\ar[r] & W^pV\ar[r]^{e} &
(P_pV^*)^*\ar[r]^{\res} & (P_{p-1}V^*)^*\ar[r] & 0
}
\end{equation}
is exact except possibly at $(P_pV^*)^*$.
For this, suppose $\phi$ is in the kernel of $\res$.  Let $x=e^{-1}(\phi)$.
That $\phi$ restricts to zero means
$f(x)=0$ for all $f\in P_{p-1}V^*$, which is to say that $x\in W^pV$,
and \eqref{eq:dualityseq} is exact.

The cokernel of $W^{p+1}V\hookrightarrow W^pV$ is $\gr^pV$, by definition.
Since $\gr_pV^*$ is 
free (Theorem~\ref{th:vg}), the kernel of the map $\res$ is $(\gr_pV^*)^*$.
It follows that the restriction of $e$ induces an isomorphism 
$\gr^pV\cong(\gr_pV^*)^*$.
\end{proof}
\subsection{Flag cochains}
The dual counterparts of the combinatorial Heaviside functions are
the {\em flag cochains} of \cite{VG87}, which we reformulate slightly 
following \cite{De02}.  First, for a ranked poset $P$, for $p\geq0$ 
let $\flag_p(P)$ be the set of $p$-flags: that is, 
chains $x_0\lx x_1\lx \cdots\lx x_p$ in $P$ where $x_i$ has rank $i$ for 
$0\leq i\leq p$.
For a fixed arrangement $\A$, let $\Fl_p=\Fl_p(\A)$
be the free abelian group on $\flag_p(L(\A))$, modulo the relations
\[
\sum_{Y\colon X_{i-1}< Y< X_{i+1}}(X_0\lx \cdots\lx X_{i-1}\lx Y\lx X_{i+1}
\cdots\lx X_p),
\]
for each flag $(X_0\lx\cdots\lx X_p)$ and index $i$, $0<i<p$.
The groups $\Fl_p$ are isomorphic to the homology groups of 
the complexified complement of $\A$: see \cite{SV91}.  In particular, 
\begin{equation}\label{eq:dualb}
\Fl_p(\A)\cong\bigoplus_{X\in L_p(\A)}\Fl_p(\A_X),
\end{equation}
a dual formulation of Brieskorn's Lemma.  The rank of the summand 
indexed by $X$ is $\abs{\mu(W,X)}$.

Let $\FF$ be a flag in $\cF(\A)$: then
$\FF=(F_0\lx F_1\lx\cdots\lx F_p)$ for some $p$, where each 
$F_i$ is a face of codimension $i$.  Since we continue to assume that
$\A$ is a central arrangement, each face $F$ has an antipodally opposite
face, which we denote by $\overline{F}$.
Define an element of $V$ using the following expression in the face algebra:
\begin{equation}\label{eq:b}
b(\FF)=F_p(F_{p-1}-\overline{F_{p-1}})(F_{p-2}-\overline{F_{p-2}})
\cdots(F_0-\overline{F_0}).
%\sum_{I\subseteq[p]}(-1)^{\abs{I}}
%F_p(\pm F_{p-1})(\pm F_{p-2})\cdots(\pm F_0),
\end{equation}
Varchenko and Gel\cprime fand~\cite{VG87} find that the flag cochains
$b(\FF)$ span $W^pV$, for each $p$.  In fact, they do so in a way
compatible with the Brieskorn decompositions \eqref{eq:brieskorn'},
\eqref{eq:dualb}.
In order to indicate how this goes, we need some additional notation.
\begin{example}\label{ex:b}
For the arrangement of Example~\ref{ex:threelines}, 
let $\FF=(\px\px\mx\lx0\px\mx\lx000)$.  
Then $b(\FF)=C_{\px\px\mx}-C_{\mx\px\mx}-C_{\px\mx\px}+C_{\mx\mx\px}$.
\end{example}

For any $\FF\in\flag_p(\cF)$, let $\abs{\FF}$ denote the flag $\XX$
in the intersection lattice with $X_i=\abs{F_i}$.  Regard $F_p$ as a 
chamber of $\A^{X_p}$, in order to define a map
\[
f\colon \flag_p(\cF(\A))\to
\bigsqcup_{X\in L_p(\A)}\flag_p(L(\A_X))\times \cC(\A^X)
\]
by $f(\FF)=(\abs{\FF},F_p)$.  Consider the fibres of $f$.
If $f(\FF)=(\XX,F)$, then $F_p=F$.  There are two possibilities for
$F_{p-1}$, however, obtained by moving away from $F$ inside $X_{p-1}$
in either of two directions.  Once $F_{p-1}$ is chosen, there are 
two possiblities for $F_{p-2}$, and so on.  By inspecting
\eqref{eq:b}, one finds that flag cochains $b(\FF)$ differ at most by
a sign on flags in the same fibre of the map $f$.

In order to reconcile the choice of signs, choose an orientation of each
face of the zonotope $Z_\A$ so that parallel faces have the same orientation,
but arbitrarily otherwise.  (This is equivalent to choosing a coorientation
of each element of $L(\A)$, as in \cite{VG87}, but easier to draw.)  
Then for each covering relation $F<G$ in $\cF(\A)$, let
$\vareps(F,G)=\pm1$ according to whether or not the orientations on $F$ and
$G$ agree.  Let $\vareps(\FF)=\prod_{i=0}^{p-1}\vareps(F_i,F_{i+1})$.

The last result we need to recall is the dual formulation of 
\cite[Theorem~8]{VG87}.
\begin{theorem}
For $0\leq p\leq \ell$, there is a well-defined map depending on the
choices of orientations,
\begin{equation}\label{eq:noncanonical}
\pi_p\colon W^p V(\A)\to \Fl_p(\A),
\end{equation}
for which $\pi_p(b(\FF))=\vareps(\FF)\abs{\FF}$, 
for all flags $\FF\in\flag_p(\cF)$.
The kernel of $\pi_p$ is $W^{p+1}V(\A)$.
\end{theorem}

There is also a map in the other direction, given by \cite[Theorem 18.3.3]{Va93}:
\begin{proposition}[\cite{Va93}]
For each $p$, the map 
\begin{equation}\label{eq:defphi}
\phi_p\colon\bigoplus_{X\in L_p(\A)}
\Fl_p(L(\A_X))\otimes_R V(\A^X)\to W^pV(\A)
\end{equation}
given by sending $\XX\otimes F$ to $\vareps(\FF)b(\FF)$ is well-defined, where
$\FF$ is any flag with $f(\FF)=(\XX,F)$.  
\end{proposition}

\begin{example}[Example~\ref{ex:b}, continued]\label{ex:c}
Orient the faces of the zonotope as shown in Figure~\ref{fig:b}.  For
the flag $\XX=(\R^2,H_1,0)$, the only choice for $F$ is the face $000$.
Picking $\FF=(\px\px\mx\lx0\px\mx\lx000)$, we see $\vareps(\FF)=-1$, and
coordinates of $\phi(\XX\otimes C_{000})$ are as indicated.  The 
simplex corresponding to $\FF$ in the barycentric subdivision of $Z_\A$ is 
shaded in Figure~\ref{fig:b1}.

For $\XX=(\R^2,H_1)$, there are two choices for $F$.  We find
$\phi(\XX\otimes C_{0\mx\px})=C_{\px\mx\px}-C_{\mx\mx\px}$ and
$\phi(\XX\otimes C_{0\px\mx})=C_{\px\px\mx}-C_{\mx\px\mx}$.
\end{example}

\begin{figure}
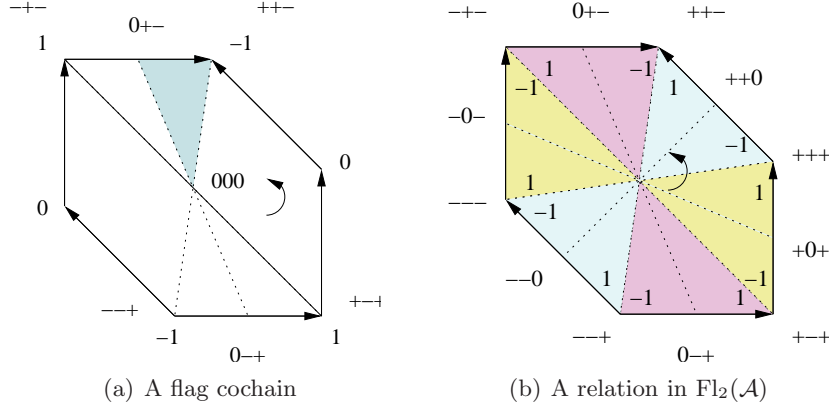

\subfigure[A flag cochain]{\includegraphics[height=1.9in]{zono4.pstex}
\label{fig:b1}}\qquad
\subfigure[A relation in $\Fl_2(\A)$]{\includegraphics[height=1.9in]{zono5.pstex}}
\caption{Flag cochains}\label{fig:b}
\end{figure}
\section{Eigenvectors for the random walk}
\subsection{The face algebra action on flag cochains}
We recall that $A=R[\cF]$ denoted the face algebra.  The domain of 
the map $\phi_p$ in \eqref{eq:defphi} has the structure of a left $A$-module
by the action of $A$ on each $V(\A^X)$.  We see in this section that the
codomain of $\phi_p$ is also an $A$-module (Corollary~\ref{cor:submodule}), and
$\phi_p$ is an $A$-module homomorphism (Lemma~\ref{lem:hom}).

First, consider
applying the face product to each element of a flag $\FF\in\flag_p(\cF)$.
Clearly for each $F\in\cF$, we have $FF_0\leq \cdots\leq FF_p$.
By \cite[Lemma~3.2]{De02}, however, 
this is a flag if and only if $\abs{F}\geq\abs{F_p}$; otherwise, 
$FF_k=FF_{k+1}$ for some $k$ with $0\leq k\leq p-1$.
As usual, we regard the faces $F$ with $\abs{F}\geq\abs{F_p}$
as chambers of $\A^\abs{F_p}$.
In this case, let $F\FF$ denote the flag $FF_0\lx\cdots\lx FF_p$.  
\begin{proposition}\label{prop:dotwithb}
For each $F\in\cF$ and $\FF\in\flag_p(\cF)$, we have
\[
Fb(\FF) = \begin{cases}
b(F\FF) & \text{if $\abs{F}\geq \abs{F_p}$;}\\
0 & \text{otherwise.}
\end{cases}
\]
\end{proposition}
\begin{proof}
By the deletion property,
\begin{eqnarray*}
Fb(\FF) &=& FF_p(F_{p-1}-\overline{F_{p-1}})(F_{p-2}-\overline{F_{p-2}})
\cdots(F_0-\overline{F_0})\\
&=&FF_p(FF_{p-1}-F\overline{F_{p-1}})(FF_{p-2}-F\overline{F_{p-2}})
\cdots(FF_0-F\overline{F_0}).\\
\end{eqnarray*}
If $\abs{F}\geq \abs{F_p}$, then $\abs{F}\geq\abs{F_i}$ for each $i$,
so $F\overline{F_i}=\overline{FF_i}$ for each $i$, and we obtain $b(F\FF)$.

Otherwise, $FF_k=FF_{k+1}$ for
some $k$.  Since $F\overline{F}=F$ for any $F$, it is easy to check
that $(FF_k-F\overline{F_k})(FF_k-F\overline{F_k})=0$, from which it follows
that $Fb(\FF)=0$ as well.
\end{proof}
Since $W^pV$ is spanned by vectors $b(\FF)$ for $\FF\in\flag_p(\cF)$, we
see the filtration is compatible with the face product.  That is,
\begin{corollary}\label{cor:submodule}
For $0\leq p\leq \ell$, $W^pV$ is an $A$-submodule of $V$.
\end{corollary}
\begin{lemma}\label{lem:hom}
The map $\phi$ of \eqref{eq:defphi} is an $A$-module homomorphism.
\end{lemma}
\begin{proof}
For $\XX\otimes F\in \Fl_p(\A_X)\otimes V(\A^X)$, choose a flag $\FF$ with
$\abs{\FF}=\XX$ and $F_p=F$.  Recall that we chose the orientations in $Z_\A$
to agree on parallel faces.  Note the face $\widehat{GF_i}$ is a translate
of $\hat{F_i}$ (Proposition~\ref{prop:Zprod}), for all $0\leq i\leq p$ and
$\abs{G}\geq\abs{F_p}$, so their orientations agree.  It follows that
$\vareps(\FF)=\vareps(G\FF)$ for all faces $G$ with $\abs{G}\geq X$.  So
for $G\in\cC(\A^X)$, 
\begin{eqnarray*}
\phi(G(\XX\otimes F)) &=&\vareps(G\FF) b(G\FF)\\
&=&\vareps(\FF)G b(\FF),\\
&=&G\phi(\XX\otimes F)
\end{eqnarray*}
using Proposition~\ref{prop:dotwithb} at the second step.  On the other
hand, if $\abs{G}\not \geq X$, both sides are zero, by \eqref{eq:localaction} and 
Proposition~\ref{prop:dotwithb}.
\end{proof}
\subsection{The main result}
Now we are able to state and prove a description of each eigenspace of
the BHR random walk's transition matrix $K$.  The main idea is to
use the stationary distribution \eqref{eq:stationary} on each 
the arrangement $\A^X$, for each $X$, together with Lemma~\ref{lem:hom}.

Recall that eigenvalues of $K$ were indexed by subspaces $X\in L(\A)$,
from \eqref{eq:lambda}.
For each $X\in L(\A)$, then let $q^X\in V(\A^X)$ be the eigenvector
\eqref{eq:1eivec} for the arrangement $\A^X$, a vector over the ring $R$: 
by Lemma~\ref{lem:eivec},
\begin{equation}\label{eq:local}
K_{\A^X}\cdot q^X=\lambda_X q^X.
\end{equation}
For each 
$C\in\cC(\A^X)$, let $q^X_C$ denote the $C$th coordinate of $q^X$.

For each $X\in L_p(\A)$, define a map $\psi_X\colon \Fl_p(\A_X)\to
W^pV(\A)\subseteq V(\A)$ by letting
\begin{equation}\label{eq:eigenvector}
\psi_X(\XX)=\sum_{C\in\cC(\A^X)}q^X_C\phi(\XX\otimes C).
\end{equation}
Let $\psi\colon\Fl(\A)\to V$ be given on $\Fl_p(\A)$ 
by composing the isomorphism
\eqref{eq:dualb} with the maps $\psi_X$.

\begin{theorem}\label{th:main}
The map $\psi_X$ is one-to-one, and its image is the eigenspace of 
$K_\A$ with eigenvalue $\lambda_X$.  
\end{theorem}

In other words, for each flag $\XX$ ending at $X$,
and each face $F$ with $\abs{F}=X$, choose any flag $\FF\in\flag(\cF)$
with $\abs{\FF}=\XX$, ending at $F$.
Then the vector
\[
\sum_{F\colon\abs{F}=X}\vareps(\FF)q^X_F b(\FF)
\]
is an eigenvector for $K$ with eigenvalue $\lambda_X$.  Moreover, we see now that the 
Varchenko-Gel\cprime fand dual filtration (\S\ref{ss:dualfiltr}) can be used
to keep track of the eigenspace multiplicities (Theorem~\ref{th:BHR}):
\begin{corollary}
The map $\psi$ is an isomorphism, taking the Brieskorn decomposition \eqref{eq:dualb} 
of $\Fl(\A)$ to the eigenspace decomposition of $V$.
\end{corollary}
\begin{proof}[Proof of Theorem~\ref{th:main}]
First check $\psi_X(\XX)$ is an eigenvector, for any flag $\XX$.  We have
\begin{eqnarray*}
K\cdot\psi(\XX) &=&\sum_{F\in\cF(\A)}w_F F\sum_{C\in\cC(\A^X)}q^X_C
 \phi(\XX\otimes C)\\
&=&\sum_{F\in\cF(\A^X)}w_F \sum_{C\in\cC(\A^X)} q^X_C\phi(\XX\otimes FC)
\\
&=&\lambda_X\psi(\XX),
\end{eqnarray*}
using first Lemma~\ref{lem:hom} then \eqref{eq:local}.

To see that the kernel of $\psi_X$ is zero, consider the two short
exact sequences
\[
\xymatrix{
0\ar[r]& W^{p+1}V(\A)\ar[r]\ar[d]_{i^*} & W^pV(\A)\ar[r]^{\pi_p}\ar[d]_{i^*} & 
\Fl_p(\A)\ar[r]\ar[d] & 0\\
0\ar[r]& W^{p+1}V(\A_X)\ar[r] & W^pV(\A_X)\ar[r]^{\pi_p} & 
\Fl_p(\A_X)\ar[r]\ar@{.>}[ul]_{\psi_X} & 0,\\
}
\]
where the rows are given by \eqref{eq:noncanonical}, and the vertical
maps are the natural ones (\S\ref{ss:dualfiltr}).  Note that
$i^*b(\FF)=i^*b(\FF')$ for any two flags $\FF$ and $\FF'\in \flag_p(\cF(\A_X))$.
It follows that 
$i^*\circ\phi(\XX\otimes C)=i^*\circ\phi(\XX\otimes C')$ for any $C,C'\in\cC(\A^X)$, so
the composite
\begin{eqnarray*}
\pi_p\circ i^*\circ \psi_X(\XX) &=& \pi_p\Big(\big(\sum_{C\in \cC(\A^X)}
q^X_C\big)i^*\phi(\XX\otimes C)\Big)\\
&=& (\sum_{C\in \cC(\A^X)}p_C^X)\XX.
\end{eqnarray*}
Since the vector $q^X\in V(\A^X)$ is a rescaling of a probability 
distribution, 
the sum of its coordinates must be nonzero.  Since our coefficient ring
$R$ is a domain, this means that the composite has zero kernel, so
$\psi_X$ is one-to-one.
\end{proof}
\begin{example}[Example~\ref{ex:c}, continued]\label{ex:d}
Here is a basis of eigenvectors in the case of three lines in the plane
given by Theorem~\ref{th:main}.
The $\lambda_{\R^2}$-eigenvector is provided by $q$ of \eqref{eq:1eivec}:
if we specialize the weights to a probability distribution, recall
$1=\lambda_{\R^2}$, and the eigenvector is the stationary distribution
\eqref{eq:stationary}.

For an arrangement of one point in a line, the vector \eqref{eq:1eivec}
equals 
\begin{equation}\label{eq:onepoint}
\frac{w_\px+w_0+w_\mx}{w_0w_\px w_\mx(w_\px+w_\mx)}\cdot(w_\px,w_\mx),
\end{equation}
ordering the chambers with $\px$ first.  
%Then
%$p^{H_1}_{0\px\mx}=w_{0\px\mx}$ and $p^{H_1}_{0\mx\px}=w_{0\mx\px}$.
Let $\XX=(\R^2,H_1)$.  Then $\Phi_{H_1}(\XX)$ is a unit multiple of 
$w_{0\px\mx}\phi(\XX\otimes C_{0\px\mx})
+w_{0\mx\px}\phi(\XX\otimes C_{0\mx\px})$,
as in \eqref{eq:onepoint}.
Using the calculation in Example~\ref{ex:c} gives the vector shown
in Figure~\ref{fig:c}.

We order the basis of $V$ counterclockwise, starting with the chamber
$\px\px\px$.  For clarity, the weights of the codimension-$0$ and 
-$1$ faces are relabelled $w_0,\ldots,w_6$ as shown in Figure~\ref{fig:c}.
Then the eigenvectors given by Theorem~\ref{th:main} are
\[
\begin{array}{c|c|c}
\XX & \Psi(\XX)\text{~proportional to} & \lambda\\ \hline
(\R^2) & p,\text{~above} & \sum_{i=0}^6 w_i\\ \hline
(\R^2,H_1) & (0,w_2,-w_2,0,-w_5,w_5) &
w_0+w_2+w_5\\
(\R^2,H_2) & (w_6,0,w_3,-w_3,0,-w_6) & w_0+w_3+w_6\\
(\R^2,H_3) & (-w_1,w_1,0,w_4,-w_4,0) & w_0+w_1+w_4\\ \hline
(\R^2,H_1,0) & (0,-1,1,0,-1,1) & w_0\\
(\R^2,H_2,0) & (1,0,-1,1,0,-1) & w_0
\end{array}
\]
\end{example}
\begin{figure}
\includegraphics[height=1.9in]{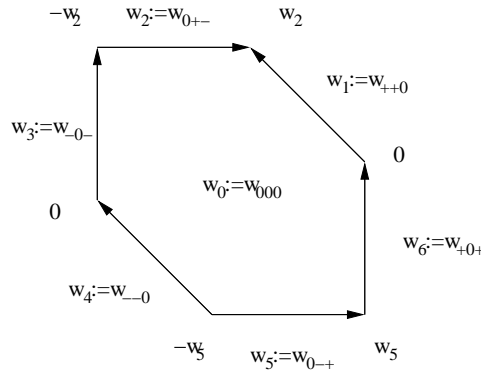}
\caption{$\Psi_{H_1}(\R^2,H_1)$ in Example~\ref{ex:d}}\label{fig:c}
\end{figure}

\begin{remark}
The results of \cite{VG87} on which Theorem~\ref{th:main} is based 
can likely be generalized to arbitrary oriented matroids without
change, in which case the results here would generalize in the same 
way: see the discussion in \cite{BD98}.  

In \cite{Br00}, Brown shows that many of the BHR random walk's properties
(such as diagonalizability) can also be generalized by replacing the face
algebra with any semigroup with the left-regular band property.  The same
paper describes idempotents for irreducible modules, which implicitly
give eigenvectors for the generalized random walk.  However, that description
is somewhat more complicated than the one here.  
Recent work of 
Saliola~\cite{sa10} gives a quite different description of the eigenvectors of 
the BHR random walk, using methods that hold for any left-regular
band.  His construction is complementary to the one given here.  The
cost of working with the general setting of \cite{Br00} is apparently
no longer to have an eigenbasis with such convenient labelling as 
non-broken circuits.  It would be interesting to know, then, if the
main results of \cite{VG87} would admit left-regular band generalizations.
\end{remark}
\begin{ack}
The author would like to thank Phil Hanlon for helpful discussions at the
start of this project, 
and the Institut de G\'eom\'etrie, Alg\`ebre et Topologie at the EPFL
for its hospitality during its completion.
\end{ack}

\bibliographystyle{amsplain}
%\bibliography{diagonal}
%\end{document}

\def\cprime{$'$}
\providecommand{\bysame}{\leavevmode\hbox to3em{\hrulefill}\thinspace}
\providecommand{\MR}{\relax\ifhmode\unskip\space\fi MR }
\providecommand{\MRhref}[2]{%
  \href{http://www.ams.org/mathscinet-getitem?mr=#1}{#2}
}
\providecommand{\href}[2]{#2}

\end{document}